\documentclass{amsart}

\usepackage{amssymb}
\usepackage{mathrsfs}

\newtheorem{theorem}{Theorem}[subsection]
\newtheorem{lemma}{Lemma}[subsection]
\newtheorem{proposition}{Proposition}[subsection]
\newtheorem{corollary}{Corollary}[subsection]

\theoremstyle{definition}
\newtheorem{definition}{Definition}[subsection]

\numberwithin{equation}{section}

\newcommand{\supp}{\mathrm{supp}\,}
\newcommand{\Autz}{\mathrm{Aut^0}\!}
\newcommand{\Hom}{\mathrm{Hom}}

\newcommand{\Gr}{\mathrm{Gr}}
\newcommand{\Bl}{\mathrm{Bl}}
\newcommand{\Diag}{\mathrm{diag}}

\newcommand{\C}{\mathbb C}
\newcommand{\Z}{\mathbb Z}
\newcommand{\p}{\mathbb P}
\newcommand{\A}{\mathbb A}

\begin{document}

\title{Automorphisms of wonderful varieties}

\author{Guido Pezzini}
\address{Departement Mathematik \\
Universit\"at Basel\\
Rheinsprung 21 \\ 
4051 Basel\\
Switzerland}
\email{Guido.Pezzini@unibas.ch}

\subjclass[2000]{14J50, 14L30, 14M17}
\date{June 31, 2008}

\begin{abstract}
Let $G$ be a complex semisimple linear algebraic group, and $X$ a wonderful $G$-variety. We determine the connected automorphism group $\Autz(X)$ and we calculate Luna's invariants of $X$ under its action.
\end{abstract}

\maketitle

\section{Introduction}
Let $X$ be a smooth complete algebraic variety over $\C$, and let $\Autz(X)$ be the connected component containing the identity of its automorphism group. It is well known that $\Autz(X)$ is an algebraic group, with Lie algebra equal to the global sections of the tangent bundle of $X$.

Let us consider this group in the framework of varieties endowed with an action of a connected reductive algebraic group $G$. The simplest case occurs when the action of $G$ on $X$ is transitive, i.e.\ we have a projective homogeneous space $X=G/P$ for $P$ a parabolic subgroup of $G$. In this setting $\Autz(X)$ is known, and it is always semisimple; it is interesting to notice that in a few cases it strictly contains the image of $G$ under the given homomorphism $G\to \Autz(X)$.

As soon as we abandon the transitivity hypothesis the situation gets much more complicated, even if we stick to quasi-homogeneous varieties, which means $X$ having an open dense $G$-orbit. For example, in general $\Autz(X)$ need not be reductive. The case of {\em toric varieties} is known since the work \cite{De70} by Demazure: here $\Autz(X)$ is completely determined using the fan of convex cones representing the toric variety $X$. The case of {\em regular varieties} is studied by Bien and Brion in \cite{BB96}; one of their results is the structure of the Lie algebra of $\Autz(X)$ as a $G$-module.

More recently, Brion has considered the case of {\em wonderful varieties}, which are regular varieties with some extra hypotheses. In \cite{Br07} he shows that $\Autz(X)$ is always semisimple in this case, that $X$ is wonderful under its action too, and describes some aspects of the action of $\Autz(X)$ on $X$. These results have also useful consequences on the broader class of {\em spherical varieties}.

In this article we determine the automorphism groups $\Autz(X)$ for any wonderful variety $X$. In particular we show that ``most often'' the image of $G$ is equal to the whole $\Autz(X)$, especially if the number of $G$-orbits on $X$ is greater than $2$. Our approach to describe these varieties uses their discrete invariants introduced by Luna in \cite{Lu01}: they separate wonderful varieties as shown by Losev in \cite{Lo07}.

We describe in details all varieties such that $\Autz(X)$ strictly contains the image of $G$, and we determine Luna's invariants of $X$ under the action of $\Autz(X)$.


\section{Invariants of wonderful varieties}\label{sect:def}
\subsection{Definitions}
Throughout this paper, $G$ will be a semisimple connected linear algebraic group over $\C$. In $G$ we fix a Borel subgroup $B$, a maximal torus $T\subset B$, and we denote 
by $S$ the corresponding set of simple roots. For simple groups, we will refer to the usual Bourbaki numbering of simple roots.

We will often denote by $L$ a Levi subgroup of some parabolic subgroup $P\supseteq B$ of $G$. The choice of $L$ is always supposed to be such that $B\cap L$ and $T\cap L$ are resp. a Borel subgroup and a maximal torus of $L$: in this way the simple roots of $L$ with respect to $B\cap L$ and $T\cap L$ are naturally identified with a subset of $S$.

More generally, whenever we have two reductive groups $\tilde G\supset G$, the choices of Borel subgroups $\tilde B$, $B$ and maximal toruses $\tilde T$, $T$ will always be such that $\tilde B\supset B$ and $\tilde T\supset T$. 

\begin{definition} \cite{Lu01}
A {\em wonderful $G$-variety} is an irreducible algebraic variety $X$ over $\C$ such that:
\begin{enumerate}
\item X is smooth and projective;
\item $G$ has an open (dense) orbit on $X$, and the complement is the union of ($G$-stable) prime divisors $X^{(1)},\ldots,X^{(r)}$ which are smooth, with normal crossings, and satisfy $\bigcap_{i=1}^r X^{(i)} \neq \emptyset$;
\item if $x,y\in X$ are such that $\left\{ i \;|\; x\in X^{(i)}\right\} = \left\{ j \;|\; y\in X^{(j)}\right\}$, then $x$ and $y$ lie on the same $G$-orbit.
\end{enumerate}
We will also use the notation $(G,X)$ instead of $X$ only, to keep track of the group $G$ we are considering. The number $r$ of $G$-stable prime divisors is the {\em rank} of $(G,X)$, and the divisors $X^{(i)}$ are called {\em boundary prime divisors}. Their union is denoted by $\partial (G,X)$, the {\em boundary} of $(G,X)$.
\end{definition}
Wonderful varieties include complete $G$-homogeneous spaces $G/P$, for $P$ a parabolic subgroup, as the special case where the rank is zero. A wonderful variety $X$ is always spherical, i.e.~a Borel subgroup has an open dense orbit on $X$: see \cite{Lu96}. The theory developed by Luna in \cite{Lu01} defines the following ``discrete invariants'':
\begin{definition}
Let $X$ be a wonderful $G$-variety. We define $\Xi(G,X)$ to be the lattice of $B$-eigenvalues ($B$-{\em weights}) of rational functions on $X$ that are $B$-eigenvectors. This lattice has a basis $\Sigma(G,X)$, whose elements are called the {\em spherical roots} of $X$, defined as the set of weights of $T$ acting on the quotient tangent space:
\[
\frac{T_zX}{T_z (G\cdot z)}
\]
where $z$ is the unique point of $X$ fixed by the Borel subgroup opposite of $B$ with respect to $T$. We define $\Delta(G,X)$ to be the set of $B$-stable but not $G$-stable prime divisors on $X$, called {\em colors}. It is a finite set, and it is equipped with a map:
\[
\rho_X\colon \Delta(G,X) \to \Hom_\Z(\Xi(G,X),\Z)
\]
defined as follows: if $D$ is a color then $\rho_{G,X}(D)$ is a functional on $\Xi(G,X)$ taking on $\gamma$ the value $\nu_D(f_\gamma)$, where $\nu_D$ is the discrete valuation on $\C(X)^*$ associated to $D$, and $f_\gamma\in\C(X)^*$ is a $B$-eigenvector whose $B$-eigenvalue is $\gamma$. It is also common to write this as a coupling: $\langle D,\gamma\rangle=\rho_{G,X}(D)(\gamma)$. Finally, we define $S^p(G,X)$ to be the set of simple roots associated to the parabolic subgroup $P(G,X)\supseteq B$ defined as the stabilizer of the open $B$-orbit of $X$.
\end{definition}
It is immediate from the definitions that any intersection of $n$ boundary prime divisors of $X$ is again a wonderful variety, of rank $r-n$ (for any $n=0,\ldots,r$). Moreover, each spherical root $\gamma_i$ can be naturally associated to a boundary prime divisor $X^{(i)}$ on $X$, namely the one such that the $T$-weights of the quotient:
\[
\frac{T_zX^{(i)}}{T_z (G\cdot z)}
\]
are precisely $\Sigma(G,X)\setminus\{\gamma_i\}$. A spherical root $\gamma_i$ can also be associated to a rank $1$ wonderful $G$-subvariety on $X$, namely:
\[
X_{(i)} = \bigcap_{j\neq i} X^{(j)}.
\]
This $X_{(i)}$ has $\gamma_i$ as its unique spherical root.

Finally, $\gamma_i$ is always a linear combination of simple roots of $G$ with non-negative coefficients; the set of simple roots whose coefficient is non-zero is called the {\em support} of $\gamma_i$. The {\em support} of a set of spherical roots is defined as the union of the supports of its elements.

It is worth noticing that these invariants obey some rather strict conditions of combinatorial nature, as discussed by Luna in \cite{Lu01}. In the rest of this paper we will sometimes use these conditions, although it will not be necessary to recall all the combinatorics that arises from the theory.

The last ingredient we need here is the following relation between colors and simple roots:
\begin{definition}
Let $X$ be a wonderful $G$-variety, $D$ one of its colors and $\alpha$ a simple root of $G$. We say that $\alpha$ {\em moves} $D$ if $D$ is not stable under the action of the minimal parabolic subgroup containing $B$ and associated to the simple root $\alpha$.
\end{definition}
With this definition, $S^p(G,X)$ is precisely the set of simple roots moving no color.
\begin{lemma}\label{lemma:moved} \cite{Lu01}
Let $\alpha$ be a simple root moving a color $D$. Suppose that $\alpha$ is not contained in the support of any spherical root. Then $D$ is moved only by $\alpha$, and $\alpha$ moves only $D$; moreover, the functional $\rho_{G,X}(D)$ is equal to $\alpha^\vee|_{\Xi(G,X)}$, where $\alpha^\vee$ is the coroot associated to $\alpha$.
\end{lemma}
Other links between functionals associated to colors and simple roots moving them will be found in section \ref{sect:morphisms}.
\begin{definition}
Let $X$ be a wonderful $G$-variety. We denote by $\A(G,X)$ the set of colors of $X$ moved by simple roots which are also spherical roots, and we call the triple $(S^p(G,X), \Sigma(G,X), \A(G,X))$ the {\em spherical system} of $X$.
\end{definition}
Results in Losev's paper \cite{Lo07} directly imply the following:
\begin{theorem}\cite{Lo07}\label{thm:uni}
The spherical systems $(S^p(G,X), \Sigma(G,X), \A(G,X))$ (here $\A(G,X)$ is considered just as an abstract finite set, endowed with the application $\rho_{G,X}|_{\A(G,X)}$) separate non-$G$-isomorphic wonderful $G$-varieties.
\end{theorem}

\subsection{Subvarieties, products and parabolic inductions} \label{ssect:prod}
Definitions and results in this subsection are a part of Luna's theory developed in \cite{Lu01}: we will omit here all the proofs.

Let $X$ be a wonderful $G$-variety of rank $r$, and consider a $G$-stable irreducible closed subvariety $Y$ of codimension $k$. Any such $Y$ is always wonderful, and equal to the intersection of $k$ border prime divisors of $X$:
\[
Y = \bigcap_{j=1}^k X^{(i_j)}
\]
for distinct $i_1,\ldots,i_k$ in $\{1, \ldots, r\}$. We can describe the spherical system of $Y$ in terms of that of $X$. First of all, $S^p(G,Y)=S^p(G,X)$. Then, the spherical roots $\Sigma(G,Y)$ are exactly $\Sigma(G,X) \setminus \{ \gamma_{i_1},\ldots,\gamma_{i_k} \}$. Finally, the set $\A(G,Y)$ is equal to $\A(G,X)$ minus all colors moved only by simple roots in $\Sigma(G,X)\setminus\Sigma(G,Y)$.

\begin{definition}
A wonderful $G$-variety $X$ is a {\em product} if $G=G_1\times G_2$ and $X=X_1\times X_2$ where $X_i$ is a wonderful $G_i$-variety for $i=1,2$. If $X$ is not a product, we say it is {\em indecomposable}.
\end{definition}
A wonderful variety is a product exactly when its associated data is a product, in the following sense:
\begin{definition} \label{def:prod}
The spherical system $(S^p(G,X), \Sigma(G,X), \A(G,X))$ is a {\em product} if $G=G_1\times G_2$, correspondingly $S=S_1 \cup S_2$ with $S_1\perp S_2$, and we have:
\begin{itemize}
\item[-] $S^p(G,X)=S^p_1\cup S^p_2$,
\item[-] $\Sigma(G,X)=\Sigma_1\cup\Sigma_2$,
\item[-] $\A(G,X)=\A_1\cup\A_2$,
\item[-] for all $D\in\A_1$, $\rho_{G,X}(D)$ is zero on $\Sigma_2$ and for all $D\in\A_2$, $\rho_{G,X}(D)$ is zero on $\Sigma_1$;
\end{itemize}
where we define: $S^p_i=S^p(G,X) \cap S_i$, $\Sigma_i=\{\gamma\in\Sigma(G,X) \;| \;\supp\gamma\subseteq S_i\}$, and $\A_i$ is the set of colors in $\A(G,X)$ moved only by simple roots in $S_i$ ($i=1,2$).
\end{definition}
A very simple example of product is where some factor has rank $0$; about this case, definition \ref{def:prod} immediately implies the following:
\begin{lemma} \label{lemma:supportprod}
Let $X$ be a wonderful $G$-variety, and suppose that $G=G_1\times G_2$. Then $X$ is a product $X_1\times X_2$ where $X_1$ is a rank $0$ wonderful $G_1$-variety if and only if $\supp\Sigma(G,X)$ doesn't contain any simple root of $G_1$.
\end{lemma}

Let $X$ be a wonderful $G$-variety, and suppose that the stabilizer $H$ of a point in the open $G$-orbit is such that $R(Q) \subseteq H \subseteq Q$ for some proper parabolic subgroup $Q$ of $G$, where $R(Q)$ is the radical of $Q$. Then $X$ is isomorphic to $G\times_Q Y$ where $Y$ is a $Q$-variety where the radical $R(Q)$ acts trivially. Moreover, $Y$ turns out to be wonderful under the action of $P/R(P)$, thus also under the action of $L$ a Levi subgroup of $Q$. Here $G\times_Q Y$ is defined as the quotient $(G\times Y)/\sim$ where $(g,x)\sim(gq,q^{-1}x)$ for all $q\in Q$.
\begin{definition}
Such a wonderful variety $X\cong G\times_Q Y$ is said to be a {\em parabolic induction} of $Y$ by means of $Q$. A wonderful variety which is not a parabolic induction is said to be {\em cuspidal}.
\end{definition}
On the combinatorial side, this corresponds to the following:
\begin{definition} \label{def:cuspidalss}
The spherical system $(S^p(G,X), \Sigma(G,X), \A(G,X))$ is said to be {\em cuspidal} if $\supp\Sigma(G,X) \cup S^p(G,X)=S$.
\end{definition}
If $X$ has no rank $0$ factor then it is cuspidal if and only if $\supp\Sigma(G,X)=S$.

Non-cuspidal wonderful varieties are often ``ignored'' due to the fact that the $G$-action on $X$ is completely determined by the $L$-action on $Y$, and the spherical system $(S^p(G,X), \Sigma(G,X), \A(G,X))$ is equal to $(S^p(L,Y), \Sigma(L,Y), \A(L,Y))$, of course up to the identification of the simple roots of $L$ with a subset of $S$.

On the other hand, the whole sets of colors of $X$ and $Y$ are different: the set $\Delta(L,Y)$ is in natural bijection with a proper subset of $\Delta(G,X)$. The remaining colors of $X$ are in bijection with the simple roots of $G$ which are not simple roots of $L$: any such simple root $\alpha$ is associated to a color $D$ of $X$, in such a way that $\alpha$ and $D$ behave as in lemma \ref{lemma:moved}.

In this article we won't leave aside non-cuspidal varieties: the automorphism groups $\Autz(X)$ and $\Autz(Y)$ may behave quite differently.

\section{Automorphism groups}
\subsection{Main theorem}
In his article \cite{Br07}, Brion proves a number of results about the automorphism groups of wonderful varieties. In particular:
\begin{theorem}\cite{Br07} \label{thm:brion}
Let $X$ be a wonderful $G$-variety, and $\tilde G$ be any closed connected subgroup of $\Autz(X)$ containing the image of $G$. Then:
\begin{enumerate}
\item $\tilde G$ is a semisimple group, and $X$ is wonderful under its action;
\item the colors of $X$ under the action of $\tilde G$ and under the image of $G$ are the same (if we fix in $\tilde G$ a Borel subgroup containing the image of $B$);
\item the boundary divisors of $X$ under the action of the full group $\Autz(X)$ are precisely the {\em fixed divisors}, i.e. those boundary divisors $X^{(i)}$ (under the action of $G$) such that $\langle D,\gamma_i\rangle<0$ for some color $D$.
\end{enumerate}
\end{theorem}
Here is our main result:
\begin{theorem} \label{thm:main}
Let $X$ be a wonderful $G$-variety, and let us suppose that $G$ acts faithfully. Then $\Autz(X)=G$, except for the indecomposable varieties listed in subsections \ref{ssect:rank0}, \ref{ssect:rank1}, \ref{ssect:rank2}, \ref{ssect:highrank}, and products of wonderful varieties involving at least one of such cases.
\end{theorem}
The theorem follows from corollary \ref{corol:products}, and subsections \ref{ssect:rank0}, \ref{ssect:rank1}, \ref{ssect:rank2}, \ref{ssect:highrank}. In these subsections we also determine the boundary and the invariants of $X$ as a wonderful variety with respect to the action of $\Autz(X)$, for all $X$ where $\Autz(X)\neq G$. For many of these varieties, explicit geometrical descriptions are provided.

The theorem can also be ``a posteriori'' reformulated in a different form, using certain smooth morphisms with connected fibers between wonderful varieties. This will be investigated in section \ref{sect:morphisms}.

Let us make a last remark about parabolic inductions and full automorphism groups: the following lemma will be useful in subsequent proofs.
\begin{lemma} \label{lemma:noncuspfixes}
Let $(G,X)$ a wonderful variety, and suppose that it is a parabolic induction obtained from a cuspidal one, call it $(L,Y)$. Let $\gamma$ be a spherical root of $(L,Y)$, let $\alpha$ be a simple root of $G$ but not of $L$, and suppose that $\alpha$ is non-orthogonal to some simple root of $\supp \gamma$. Then the border prime divisor of $(G,X)$ associated to $\gamma$ is fixed.
\end{lemma}
\begin{proof}
The spherical root $\gamma$ is a linear combination of the simple roots of its support, with positive coefficients. On the other hand, $\alpha$ is associated to some color $D$ of $(G,X)$, in such a way that $\rho_X(D)=\alpha^\vee$: see the end of subsection \ref{ssect:prod}. This implies that $D$ is non-positive on any simple root, and hence it is strictly negative on $\gamma$.
\end{proof}

\subsection{Rank 0}
\label{ssect:rank0}
Let $X$ be a wonderful $G$-variety of rank $0$, i.e.\ a homogeneous space $G/P$ where $P$ be a parabolic subgroup of $G$. Let us also suppose that $G$ acts faithfully, hence in particular $G$ is adjoint. In \cite{De77} Demazure shows that $\Autz(G/P)$ is always equal to $G$, except for a few cases called {\em exceptional}. If $G$ is simple, the only exceptional cases are:
\begin{itemize}
\item[$\mathbf{1}_{\mathit{rk}=0}$] $G=\mathsf{PSp}_{2n}$ ($n\geq 2$), $P=$ the stabilizer of a point of $\p^{2n-1}$: here $\Autz(X)=\mathsf{PSL}_{2n}$ and $G/P\cong \Autz(X)/P'\cong\p^{2n-1}$ where $P'$ is the stabilizer in $\Autz(X)$ of a point of $\p^{2n-1}$;
\item[$\mathbf{2}_{\mathit{rk}=0}$] $G=\mathsf{PSO}_{2n+1}$ ($n\geq2$), $P=$ the stabilizer of an isotropic $n$-dimensional subspace of $\C^{2n+1}$: if we choose a suitable extension to $\C^{2n+2}$ of the symmetric bilinear form defining $\mathsf{PSO}_{2n+1}$, then $\Autz(X)=\mathsf{PSO}_{2n+2}$ and $P'$=the stabilizer of an isotropic $(n+1)$-dimensional subspace of $\C^{2n+2}$; the quotients $\Autz(X)/P'$ and $G/P$ are isomorphic via $F\mapsto F\cap \C^{2n+1}$;
\item[$\mathbf{3}_{\mathit{rk}=0}$] $G=\mathsf G_2$, $P=$ the stabilizer of $[v]\in \p(V)$, where $v$ is a primitive vector in the irreducible $7$-dimensional $\mathsf G_2$-module $V$; the latter can be seen as the complexification of the real vector space of pure octonions: here $\Autz(X)=\mathsf{PSO}_7$ and it stabilizes the norm of octonions.
\end{itemize}
If $G$ is not simple, then the exceptional cases occur exactly when one of the simple components $G_i$ of $G$ (and the corresponding parabolic subgroup $P_i$ of $G_i$) appear as one the three cases above.

Now let $X=G/P$ be exceptional, with $G$ simple: $X$ is a wonderful variety of rank zero both under $\Autz(X)$ and $G$. Hence $\Xi(G,X)=\Xi(\Autz(X),X)=\{0\}$ and $\Sigma(G,X)=\Sigma(\Autz(X),X)=\emptyset$, thus $\A(G,X)=\A(\Autz(X),X)=\emptyset$; the whole set of colors $\Delta(\Autz(X),X)$ is of course in bijection with $S\setminus S^p(\Autz(X),X)$. The only invariant to be calculated is $S^p(\Autz(X),X)$:
\[
\begin{array}{lrclrcl}
\mathbf{1}_{\mathit{rk}=0} & G \!\!\!&=&\!\!\!\mathsf{\mathsf{PSp}}_{2n} \; (n\geq 2) & S^p(G,X)\!\!\!&=&\!\!\!\{\alpha_2,\ldots,\alpha_n\} \\
 & \Autz(X)\!\!\!& =& \!\!\!\mathsf{PSL}_{2n} &  S^p(\Autz(X),X)\!\!\!&=&\!\!\!\{\alpha_2,\ldots,\alpha_{2n-1}\} \\[10pt]
\mathbf{2}_{\mathit{rk}=0} & G\!\!\!&=&\!\!\!\mathsf{PSO}_{2n+1} \; (n\geq2) & S^p(G,X)\!\!\!&=&\!\!\!\{\alpha_1,\ldots,\alpha_{n-1}\} \\
 & \Autz(X)\!\!\!&=&\!\!\!\mathsf{PSO}_{2n+2} & S^p(\Autz(X),X)\!\!\!&=&\!\!\!\{\alpha_1,\ldots,\alpha_n\} \\[10pt]
\mathbf{3}_{\mathit{rk}=0} & G\!\!\!&=&\!\!\!\mathsf G_2 & S^p(G,X)\!\!\!&=&\!\!\!\{\alpha_2\}  \\
 & \Autz(X)\!\!\!&=&\!\!\!\mathsf{PSO}_7 & S^p(\Autz(X),X)\!\!\!&=&\!\!\!\{\alpha_2,\alpha_3\}
\end{array}
\]
It will be useful to notice that in all three cases $S^p(G,X)$ contains all simple roots of $G$ except for only one.

\subsection{Borders and automorphisms} \label{ssect:genres}
It is natural to ask how ``strong'' is the border as an invariant to study the full automorphism group. The rank $0$ exceptional varieties show that we can have different groups acting with same border (empty in this case), and thus same orbits, on the same wonderful variety. This can happen in higher rank too, as we will see in this subsection; however as soon as one of the groups is $\Autz(X)$ the picture is quite simple.

In the proof of the following useful proposition, we will refer to the classification of cuspidal indecomposable rank $1$ (\cite{Ah83}, \cite{HS82}, \cite{Br89}) and rank $2$ wonderful varieties (\cite{Wa96}). A list of all of them, including their invariants, can be found in \cite{Wa96}.
\begin{proposition} \label{prop:bordo}
Let $X$ be a wonderful $G$-variety where $G$ acts faithfully, and suppose $\partial(G,X)=\partial(\Autz(X),X)$. Then either $G=\Autz(X)$, or $X$ is a product where at least one of the factors is one of the exceptional rank $0$ cases listed in subsection \ref{ssect:rank0}.
\end{proposition}
\begin{proof}
The case of rank $0$ is already done in subsection \ref{ssect:rank0}. So we suppose $X$ of rank at least $1$, and let us suppose that $G\neq \Autz(X)$.

The closed $G$-orbit $Z$ of $X$ is the intersection of all $G$-boundary divisors. Since $\partial(G,X)=\partial(\Autz(X),X)$, $Z$ is also the closed $\Autz(X)$-orbit. The adjoint groups of $G$ and of $\Autz(X)$ are obviously different, and it is not difficult to see that they both act faithfully on $Z$: so $Z$ must be an exceptional rank $0$ variety.
 
In other words $G$ splits into a product $G=G_1\times G_2$ (with $G_2$ possibly trivial), and correspondingly $Z=Z_1\times Z_2$, in such a way that $(G_1,Z_1)$ appears as one of the three cases of subsection \ref{ssect:rank0} (actually, the adjoint group of $G_1$, but this is irrelevant in what follows).

If $(G,X)$ itself splits as a product $(G_1,X_1)\times (G_2,X_2)$, where $(G_1,X_1)$ has rank zero, then $X_1=Z_1$ and we are done: $(G_1,X_1)$ will be an exceptional rank $0$ factor of $X$. Moreover, for this to happen it is sufficient that $\supp \Sigma(G,X)$ contains no simple roots of $G_1$, see lemma \ref{lemma:moved}.

If $(G,X)$ doesn't split as such a product, there exist at least one spherical root, call it $\gamma_1$, whose support contains some simple root of $G_1$. Recall its associated rank $1$ wonderful $G$-subvariety:
\[
X_{(1)}=\bigcap_{j\neq 1} X^{(j)}.
\]
It has only one spherical root, namely $\gamma_1$, and of course the same closed orbit $Z$ as $X$.

Like any wonderful variety of rank $1$, $(G,X_{(1)})$ can be obtained from a cuspidal, indecomposable rank $1$ variety $(L,X_{(1)}')$ using parabolic inductions and products by rank $0$ wonderful varieties.

But $\supp\gamma_1$ contains some simple root of $G_1$: hence the factor $(G_1,Z_1)$ of $(G,Z)$ cannot ``appear'' after a product by a rank $0$ variety. Moreover, $S^p(G_1,Z_1)$ contains all simple roots of $G_1$ except one: this implies that $(G_1,Z_1)$ cannot appear after a parabolic induction. We conclude that $G_1$ was already a factor of $L$, and $Z_1$ was already a factor of $Z'$ the closed orbit of $X_{(1)}'$.

Now we look at the list in \cite{Wa96} of cuspidal indecomposable rank $1$ wonderful varieties. Our $(L,X_{(1)}')$ must appear there, and recall that $(G_1,Z_1)$ is a simple exceptional factor of its closed orbit $(L,Z')$. This doesn't match the varieties in the list, except for the following particular situation: we must have that $G_1=L$, that $(G,X_{(1)})$ is a product where the only rank $1$ factor is $(G_1,X_{(1)}')$, and that $(G_1,X_{(1)}')$ is equal to one of four candidates: cases 10, 11, 13, 14 of \cite{Wa96} (see below for details). More precisely, cases 10 or 11 could occur if $(G_1,Z_1)=\mathbf{2}_{\mathit{rk}=0}$, and cases 13 or 14 if $(G_1,Z_1)=\mathbf{3}_{\mathit{rk}=0}$. No case could occur if $(G_1,Z_1)=\mathbf{1}_{\mathit{rk}=0}$.

Then, we look at the classification in rank $2$ presented in \cite{Wa96}: none of our four candidates appears as a rank $1$ subvariety of a cuspidal indecomposable rank $2$ variety. As a consequence, $(G_1,X_{(1)}')$ cannot be a $G$-subvariety (nor a factor of a $G$-subvariety) of a higher rank wonderful variety unless the latter is a product where one of the factors is $(G_1,X_{(1)}')$. This applies of course to $(G,X)$, so we have that $(G,X)$ is itself a product where $(G_1,X_{(1)}')$ is one of the factors. By the way, this implies that $\Autz(X_{(1)}')$ is naturally a subgroup of $\Autz(X)$.

Finally, for all the four candidates for $(G_1,X_{(1)}')$, the border divisor is not fixed. But $(G_1,X_{(1)}')$ is a factor of $(G,X)$, so at least one of the border divisors of $(G,X)$ is not fixed: this is absurd since we were supposing that $\partial(G,X)=\partial(\Autz(X),X)$.
\end{proof}
As we said, the proposition fails if we replace $\Autz(X)$ with some connected group $\tilde G$ strictly between $\Autz(X)$ and $G$: the above proof suggests the rank $1$ cases 10, 11, 13, 14 of \cite{Wa96} as counterexamples.

If $(G,X)$ is one of them, then $X$ is homogeneous under $\Autz(X)$, but there exist an intermediate connected group $G\subset \tilde G \subset \Autz(X)$ such that $X$ is wonderful both under $G$ and under $\tilde G$ with the same orbits. However, a proof similar to proposition \ref{prop:bordo} shows that these are the only indecomposable wonderful varieties (regardless of the rank, provided it is $\geq 1$) having this property.

Let us see these cases in details. For $(G,X)$ equal to cases 10, 11, 13, 14, we have that $(\tilde G,X)$ is equal resp. to cases 5$\mathsf D$, 6$\mathsf D$, 7$\mathsf B$, 8$\mathsf B$. More precisely:
\begin{itemize}
\item[11-6$\mathsf D$]:\hspace{5pt} $X=\p^7$, $\partial (G,X)=\partial (\tilde G,X)$ is a smooth quadric, and:
\[
\begin{array}{cc}
\begin{array}{rcl}
G \!\!\!& = & \!\!\!\mathsf{PSO}_7 \\
S^p(G,X) \!\!\!&=& \!\!\!\{ \alpha_1, \alpha_2 \}\\
\Sigma(G,X) \!\!\!&=& \!\!\!\left\{  \alpha_1 + 2\alpha_2 + \alpha_3 \right\} \\
\A(G,X) \!\!\! &=& \!\!\! \emptyset \\
\end{array}
&
\begin{array}{rcl}
\tilde G \!\!\!& = & \!\!\!\mathsf{PSO}_8 \\
S^p(\tilde G,X) \!\!\!&=& \!\!\!\{ \alpha_1, \alpha_2 \}\\
\Sigma(\tilde G,X) \!\!\!&=& \!\!\!\left\{  2\alpha_1 + 2\alpha_2 + \alpha_3 + \alpha_4 \right\} \\
\A(\tilde G, X) \!\!\! &=& \!\!\! \emptyset \\
\end{array}
\end{array}
\]
\item[14-8$\mathsf B$]:\hspace{5pt} $X=\p^6$, $\partial (G,X)=\partial (\tilde G,X)$ is a smooth quadric, and:
\[
\begin{array}{cc}
\begin{array}{rcl}
G \!\!\!& = & \!\!\!\mathsf{G}_2 \\
S^p(G,X) \!\!\!&=& \!\!\!\{ \alpha_2 \} \\
\Sigma(G,X) \!\!\!&=&\!\!\! \left\{  4 \alpha_1 + 2 \alpha_2 \right\} \\
\A(G, X) \!\!\! &=& \!\!\! \emptyset \\
\end{array}
&
\begin{array}{rcl}
\tilde G \!\!\!& = & \!\!\!\mathsf{PSO}_7 \\
S^p(\tilde G,X) \!\!\!&=& \!\!\!\{ \alpha_2, \alpha_3 \}\\
\Sigma(\tilde G,X) \!\!\!&=&\!\!\! \left\{  2\alpha_1 + 2\alpha_2 + 2\alpha_3 \right\} \\
\A(\tilde G, X) \!\!\! &=& \!\!\! \emptyset \\
\end{array}
\end{array}
\]
\end{itemize}
Case 10-5$\mathsf D$ (resp. 13-7$\mathsf B$) is a smooth quadric of dimension 7 (resp. 6), a 2:1 cover of case 11-6$\mathsf D$ (resp. 14-8$\mathsf B$). The group $\tilde G$ is $\mathsf{SO_8}$ (resp. $\mathsf{PSO_7}$) and all the invariants are the same of case 11-5$\mathsf D$ (resp. 14-8$\mathsf B$), except that the spherical root is always one half of the spherical root of case 11-5$\mathsf D$ (resp. 14-8$\mathsf B$), under both the actions of $G$ and $\tilde G$.

\begin{corollary} \label{corol:products}
If a wonderful variety $(G,X)$ is a product of two wonderful varieties $(G_1,X_1)$, $(G_2,X_2)$, then $\Autz(X)=\Autz(X_1)\times \Autz(X_2)$.
\end{corollary}
\begin{proof}
The wonderful variety $(\Autz(X_1)\times \Autz(X_2),X)$ is of course the product of $(\Autz(X_1),X_1)$ and $(\Autz(X_2),X_2)$. All boundary prime divisors of $(\Autz(X_1),X_1)$ and of $(\Autz(X_2),X_2)$ are fixed (theorem \ref{thm:brion}), and from the analysis of the invariants associated to a product this implies that all boundary prime divisors of $(\Autz(X_1)\times \Autz(X_2),X)$ are fixed. In other words, we have $\partial(\Autz(X),X)=\partial(\Autz(X_1)\times \Autz(X_2),X)$. Now our corollary is an easy consequence of proposition \ref{prop:bordo} applied to the wonderful variety $(\Autz(X_1)\times \Autz(X_2),X)$.
\end{proof}
This corollary allows us to deal, from now on, with indecomposable varieties only.

\subsection{Rank 1}\label{ssect:rank1}
\begin{proposition} \label{prop:rank1}
Let $X$ be a indecomposable wonderful $G$-variety of rank $1$ where $G$ acts faithfully. Then $G=\Autz(X)$ if and only if $X$ is not cuspidal, or $G=\mathsf G_2$ and $X$ is the only rank $1$ wonderful $G$-variety having spherical root $\alpha_1+\alpha_2$ (case 15 of \cite{Wa96}).
\end{proposition}
\begin{proof}
Theorem \ref{thm:brion} implies that if a rank $1$ variety $(G,X)$ has a non-fixed border prime divisor, then $\Autz(X)$ strictly contains $G$, and $X$ is homogeneous under the action of $\Autz(X)$. This regards all the cuspidal indecomposable rank $1$ varieties of \cite{Wa96} except one: case 15 (in the notations of the cited paper). On the other hand if $(G,X)$ is equal to case 15, it is not homogeneous under $\Autz(X)$ thanks to the same theorem: its border prime divisor is fixed, hence $\partial(G,X)=\partial(\Autz(X),X)$ and proposition \ref{prop:bordo} assures that $\Autz(X)=G$ in this case.

Now we turn to the non-cuspidal case. Any non-cuspidal indecomposable rank $1$ variety $(G,X)$ is obtained by parabolic induction from a cuspidal one. The latter might be a product, however it will have only one cuspidal rank $1$ factor, call it $(L,X')$. It is very easy to see from lemma \ref{lemma:supportprod} that in order to obtain an indecomposable $(G,X)$, some simple root of $G$ but not of $L$ must be non-orthogonal to some simple root of $L$.

Since $(L,X')$ is cuspidal and indecomposable of rank $1$, $\supp\Sigma(L,X')$ is the whole set of simple roots of $L$. Lemma \ref{lemma:noncuspfixes} assures in this case that the border prime divisor of $(G,X)$ is fixed, and $\Autz(X)=G$ thanks to proposition \ref{prop:bordo}.
\end{proof}
Let us describe $(\Autz(X),X)$ for all rank $1$ varieties $(G,X)$ that are homogeneous under $\Autz(X)$. In \cite{Ah83} this is done for all $X$ having an affine open $G$-orbit, so we will work on the remaining ones, namely cases $9\mathsf B$ and $9\mathsf C$ of \cite{Wa96}.
\begin{itemize}
\item[$9\mathsf B$] Here $G=\mathsf{PSO}_{2n+1}$ ($n\geq 2$), $\Sigma(G,X)=\{\gamma_1=\alpha_1+\cdots+\alpha_n\}$, $S^p(G,X)=\{\alpha_2,\ldots,\alpha_{n-1}\}$,
$\A(X)=\emptyset$.
The generic stabilizer (described in \cite{Wa96}) has a Levi component isomorphic to $\mathsf{GL}_n$ and the unipotent radical isomorphic to $\bigwedge^2 \C^n$. If $\omega$ is the symmetric bilinear form on $\C^{2n+1}$ defining $G$, the open $G$-orbit of $X$ is isomorphic to:
\[
\left\{ (E,F) \in \Gr(2n, \C^{2n+1})\times \Gr(n, \C^{2n+1}) \left|
\begin{array}{ll}
\omega|_E \textrm{ nondegenerate},\\
F \textrm{ isotropic}, \; F \subset E^\perp
\end{array}\right.
\right\}.
\]
If we ignore the condition ``$\omega|_E$ nondegenerate'' we obtain the whole $X$. Therefore, with a suitable extension of $\omega$ to $\C^{2n+2}$, $X$ is isomorphic to:
\[
\left\{ (E',F') \in \Gr(2n+1, \C^{2n+2})\times \Gr(n+1, \C^{2n+2}) \left|
\begin{array}{ll}
F' \textrm{ isotropic},\\
F' \subset (E')^\perp 
\end{array} \right.
\right\}
\]
where the isomorphism is given by $(E',F')\mapsto (E'\cap \C^{2n+1}, F'\cap \C^{2n+1})$. Now it is evident that $X$ is homogeneous under $\Autz(X)=\mathsf{PSO}_{2n+2}$, and $S^p(\Autz(X),X)=\{ \alpha_2, \ldots, \alpha_n \}$.
\item[$9\mathsf C$] Here $G=\mathsf{PSp}_{2n}$ ($n\geq 2$), $\Sigma_G(X)=\{\gamma_1=\alpha_1+2\alpha_2+\cdots+2\alpha_{n-1}+\alpha_n\}$, $S^p(G,X)=\{\alpha_3,\ldots,\alpha_n\}$,
$\A(X)=\emptyset$.
The generic stabilizer has a Levi component isogenous to $\mathsf{Sp}_{2n-2}\times \C^\times$ and the unipotent radical isomorphic to $\C$. If $\omega$ is the skew-symmetric bilinear form defining $G$, the open $G$-orbit of $X$ is isomorphic to:
\[
\left\{ (E,F) \in \Gr(2, \C^{2n})\times \p(\C^{2n}) \left|
\begin{array}{ll}
E \textrm{ nonisotropic},\\
F \subset E
\end{array}\right.
\right\}.
\]
If we ignore the condition ``$E$ nonisotropic'' we obtain the whole $X$. So $X$ is omogeneous under $\Autz(X)=\mathsf{PSL}_{2n}$, and we have $S^p(\Autz(X),X)=\{ \alpha_3, \ldots, \alpha_{2n-1} \}$.
\end{itemize}

\subsection{Rank 2} \label{ssect:rank2}
\begin{proposition} \label{prop:rank2}
Let $X$ be an indecomposable $G$-wonderful variety of rank $2$ where $G$ acts faithfully. If $X$ is cuspidal then $G=\Autz(X)$ except for the following cases: 
\begin{enumerate}
\item[$\mathbf{1}_{\mathit{rk}=2}$] $G=\mathsf{PSL}_2\times \mathsf{PSp}_{2n}$ ($n\geq2$), $X$ is case 1 of type $\mathsf C$ of \cite{Wa96}; in particular:
\begin{eqnarray*}
S^p(G,X) \!\!\!&=& \!\!\!\{ \alpha'_3,\ldots,\alpha'_n \}\\
\Sigma(G,X) \!\!\!&=& \!\!\!\{ \gamma_1=\alpha_1+\alpha'_1, \gamma_2 = \alpha'_1+2\alpha'_2+\cdots+2\alpha'_{n-1}+\alpha'_n \} \\
\A(G,X) \!\!\! &=& \!\!\! \emptyset 
\end{eqnarray*}
\item[$\mathbf{2}_{\mathit{rk}=2}$] $G=\mathsf{PSp}_{2n}$ ($n\geq2$), $X$ is the first of the two cases 5 of type $\mathsf C$ of \cite{Wa96}; in particular:
\begin{eqnarray*}
S^p(G,X) \!\!\!&=& \!\!\!\{ \alpha_3,\ldots,\alpha_n \}\\
\Sigma(G,X) \!\!\!&=& \!\!\!\{ \gamma_1=2\alpha_1, \gamma_2 = \alpha_1+2\alpha_2+\cdots+2\alpha_{n-1}+\alpha_n \} \\
\A(G,X) \!\!\! &=& \!\!\! \emptyset 
\end{eqnarray*}
\item[$\mathbf{3}_{\mathit{rk}=2}$] $G=\mathsf{PSp}_{2n}$ ($n\geq2$), $X$ is the second of the two cases 5 of type $\mathsf C$ of \cite{Wa96}; in particular:
\begin{eqnarray*}
S^p(G,X) \!\!\!&=&\!\!\! \{ \alpha_3,\ldots,\alpha_n \}\\
\Sigma(G,X) \!\!\!&=& \!\!\!\{ \gamma_1=\alpha_1, \gamma_2 = \alpha_1+2\alpha_2+\cdots+2\alpha_{n-1}+\alpha_n \} \\
\A(G,X) \!\!\! &=& \!\!\! \{ D^+_1, D^-_1 \} \\
\rho_{G,X}\!\!\!  &\colon &\!\!\!
\left\{
\begin{array}{ll} 
\langle D^+_1, \gamma_1 \rangle = 1 & \langle D^-_1, \gamma_1 \rangle = 1 \\ 
\langle D^+_1, \gamma_2 \rangle = 0 & \langle D^-_1, \gamma_2 \rangle = 0  
\end{array}
\right.
\end{eqnarray*}
\item[$\mathbf{4}_{\mathit{rk}=2}$] $G=\mathsf{PSO}_{9}$, $X$ is case 3 of type $\mathsf B$ of \cite{Wa96}; in particular:
\begin{eqnarray*}
S^p(G,X) \!\!\!&=&\!\!\! \{ \alpha_2,\alpha_3 \}\\
\Sigma(G,X) \!\!\!&=& \!\!\!\{ \gamma_1=\alpha_2+2\alpha_3+3\alpha_4, \gamma_2 = \alpha_1+\alpha_2+\alpha_3+\alpha_4 \} \\
\A(G,X) \!\!\! &=& \!\!\! \emptyset 
\end{eqnarray*}
\end{enumerate}
If $X$ is not cuspidal, then $G=\Autz(X)$ except for:
\begin{enumerate}
\item[$\mathbf{5}_{\mathit{rk}=2}$] any indecomposable cuspidal variety $(G,X)$ obtained by parabolic induction from $(L,Y)$, where we require that $(L,Y)$ is a product (of possibly only one factor) having a rank $2$ factor equal to the cuspidal case 1 above, and that $G=G_1\times \mathsf{PSp}_{2n}$, i.e. the component $\mathsf{PSp}_{2n}$ of the group $L$ is preserved when passing from $L$ to $G$.
\end{enumerate}
\end{proposition}
\begin{proof}
Let us begin with the cuspidal case. Varieties $\mathbf{1}_{\mathit{rk}=2}$, $\mathbf{2}_{\mathit{rk}=2}$, $\mathbf{3}_{\mathit{rk}=2}$, $\mathbf{4}_{\mathit{rk}=2}$ are exactly the indecomposable cuspidal rank $2$ varieties where some border divisor is not fixed, following the tables in \cite{Wa96}. All other varieties satisfy $\partial(G,X)=\partial(\Autz(X),X)$ thanks to theorem \ref{thm:brion}, and $\Autz(X)=G$ thanks to proposition \ref{prop:bordo}.

In the non-cuspidal case, let $(G,X)$ be a non-cuspidal indecomposable rank $2$ variety, obtained by parabolic induction from a cuspidal one, say $(L,Y)$. We can always consider $(L,Y)$ as a product, of possibly only one factor. Anyway, it has rank $2$, like $(G,X)$.

This means that the cuspidal indecomposable factors of $(L,Y)$ of rank $>0$ are either one of rank $2$, or two of rank $1$. In the latter case the situation is similar to the second part of the proof of proposition \ref{prop:rank1} (the ``non-cuspidal part''): we discover in the same way that one of the two border divisors of $(G,X)$ can be non-fixed only if $(G,X)$ itself is a product where one of the factors has rank $1$ and non-fixed border divisor. This is absurd, since we are supposing $(G,X)$ indecomposable.

We are left with the case where $(L,Y)$ has a rank $2$ cuspidal indecomposable factor $(L_1,Y_1)$. It is evident that if the border divisors of $(L_1,Y_1)$ are all fixed, then the same happens for $(G,X)$, so we may suppose that $(L_1,Y_1)$ has at least one non-fixed divisor. Then $(L_1,Y_1)$ is equal to $\mathbf{1}_{\mathit{rk}=2}$, $\mathbf{2}_{\mathit{rk}=2}$, $\mathbf{3}_{\mathit{rk}=2}$ or $\mathbf{4}_{\mathit{rk}=2}$.

For $\mathbf{2}_{\mathit{rk}=2}$, $\mathbf{3}_{\mathit{rk}=2}$ and $\mathbf{4}_{\mathit{rk}=2}$, the group $L_1$ is simple and the spherical root corresponding to the non-fixed divisor has support equal to the whole set of simple roots of the group. From lemma \ref{lemma:noncuspfixes} we deduce that this divisor is non-fixed in $(G,X)$ only if the factor $L_1$ of $L$ is also a factor of $G$. So we have a decomposition $G=L_1\times G_1$, where no spherical root of $X$ has support on $G_1$: lemma \ref{lemma:supportprod} implies that $(G,X)$ is either a product or equal to $(L_1,Y_1)$, and this is absurd, $(G,X)$ being non-cuspidal and indecomposable.

Therefore the only situation where $(G,X)$ has a non-fixed divisor occurs if $(L_1,Y_1)=\mathbf{1}_{\mathit{rk}=2}$, where the non-fixed divisor corresponds to a spherical root whose support is the set of simple roots of $\mathsf{Sp}_{2n}$. This divisor remains non-fixed in $(G,X)$ if and only if the factor $\mathsf{Sp}_{2n}$ of $L_1$ is also a factor of $G$, and this completes the proof.
\end{proof}
Let us discuss the cuspidal varieties of the proposition above. In each case, only one border divisor is fixed: it is a rank $1$ wonderful variety which is homogeneous under its automorphism group, and which coincides with the closed $\Autz(X)$-orbit on $X$. Moreover, it always appears in full details in subsection \ref{ssect:rank1} as case $9\mathsf B$ or $9\mathsf C$. 

Using the knowledge of this divisor and its invariants under the action of its automorphism group, it is immediate to deduce the automorphism groups and the relative invariants of our rank $2$ varieties. An explicit geometrical description is also possible for the first three varieties.
\begin{enumerate}
\item[$\mathbf{1}_{\mathit{rk}=2}$] $\Autz(X)= \mathsf{PSL_2}\times\mathsf{PSL_{2n}}$, and $X$ is the following variety:
\[
X = \left\{ (E,M)  
\left|
\begin{array}{ll}
E \in \Gr(n, \C^{2n}), \\
M \in \p(\Hom(\C^2,E))
\end{array} \right.
\right\}
\]
where $\Hom(\C^2,E)$ is the space of linear homomorphisms between a fixed $\C^2$ and the $2$-dimensional space $E$. The action of $G$ and of $\Autz(X)$ are defined in the same way. The factors $\mathsf{PSp_{2n}}\subset G$ and $\mathsf{PSL_{2n}}\subset \Autz(X)$ act in the usual way on the Grassmannian; an element $(x,y)$ (where $x\in\mathsf{SL_2}$ and $y\in\mathsf{Sp_{2n}}$ or $\mathsf{SL_{2n}}$) act on the coordinate ``$M$'' as:
\[
\begin{array}{ccc}
\p(\Hom(\C^2,E))   & \to      & \p(\Hom(\C^2, yE)) \\[0pt]
M=[f]                & \mapsto  & [y \circ f \circ x^{-1}].
\end{array}
\]
The invariants under the action of $\Autz(X)$ are:
\begin{eqnarray*}
S^p(\Autz(X),X) \!\!\!&=& \!\!\!\{  \alpha'_3,\ldots,\alpha'_{2n-1} \}\\
\Sigma(\Autz(X),X) \!\!\!&=& \!\!\!\{ \gamma_1 = \alpha_1+\alpha'_1 \} \\
\A(\Autz(X),X) \!\!\! &=& \!\!\!\emptyset
\end{eqnarray*}
\item[$\mathbf{2}_{\mathit{rk}=2}$] $\Autz(X)= \mathsf{PSL_{2n}}$, $X=\Bl_{\Diag(\p^{2n-1})}((\p^{2n-1}\times\p^{2n-1})/\sim)$, where $(x,y)\sim(y,x)$,
\begin{eqnarray*}
S^p(\Autz(X),X) \!\!\!&=& \!\!\!\{ \alpha_3,\ldots,\alpha_{2n-1} \}\\
\Sigma(\Autz(X),X) \!\!\!&=& \!\!\!\{ \gamma_1=2\alpha_1 \} \\
\A(\Autz(X),X) \!\!\! &=& \!\!\! \emptyset 
\end{eqnarray*}
\item[$\mathbf{3}_{\mathit{rk}=2}$] $\Autz(X)= \mathsf{PSL_{2n}}$, $X=\Bl_{\Diag(\p^{2n-1})}(\p^{2n-1}\times\p^{2n-1})$,
\begin{eqnarray*}
S^p(\Autz(X),X) \!\!\!&=& \!\!\!\{ \alpha_3,\ldots,\alpha_{2n-1} \}\\
\Sigma(\Autz(X),X) \!\!\!&=& \!\!\!\{ \gamma_1=\alpha_1 \} \\
\A(\Autz(X),X) \!\!\! &=& \!\!\! \{ D_1^+, D_1^- \}, \;\;\; \langle D_1^+, \gamma_1\rangle = \langle D_1^-,\gamma_1\rangle=1
\end{eqnarray*}
\item[$\mathbf{4}_{\mathit{rk}=2}$] $\Autz(X)= \mathsf{PSO}_{10}$; here $X$ under the action of $\Autz(X)$ is a parabolic induction from $(L,Y)$, where $L=\mathsf{PSO_{8}}$ and $(L,Y)$ is the rank $1$ variety $\p^7$ (cf. rank $1$ case 6$\mathsf D$). The parabolic induction is given identifying $L$ with a Levi part of the parabolic subgroup of $\mathsf{PSO_{10}}$ associated to the simple roots $\alpha_2, \alpha_3, \alpha_4, \alpha_5$. The invariants under the action of $\Autz(X)$ are:
\begin{eqnarray*}
S^p(\Autz(X),X) \!\!\!&=& \!\!\!\{  \alpha_2,\alpha_3,\alpha_4 \}\\
\Sigma(\Autz(X),X) \!\!\!&=& \!\!\!\{ \gamma_1=\alpha_2+2\alpha_3+\alpha_4+2\alpha_5 \} \\
\A(\Autz(X),X) \!\!\! &=& \!\!\! \emptyset
\end{eqnarray*}
\end{enumerate}
In the end, let us consider the non-cuspidal exception $\mathbf{5}_{\mathit{rk}=2}$. Here our $(G,X)$ is a parabolic induction from $(L,Y)$ where $(L,Y)$ is the first cuspidal case, and $G=G_1\times \mathsf{PSp}_{2n}$. Again, it is easy to prove that $\Autz(X)=G_1\times \mathsf{PSL}_{2n}$ and the invariants behave, {\em mutatis mutandis}, in the same way as for $\mathbf{1}_{\mathit{rk}=2}$.

\subsection{Higher rank} \label{ssect:highrank}
\begin{proposition} \label{prop:highrank}
Let $X$ be an indecomposable $G$-wonderful variety of rank at least $3$, and let us suppose that $G$ acts faithfully. Then $G=\Autz(X)$ except for any variety $(G,X)$ such that $G$ has a factor $\mathsf{PSp}_{2n}$ with associated simple roots $\alpha_1,\ldots,\alpha_n$, and the invariants of $X$ satisfy:
\begin{eqnarray*}
\Sigma(G,X) \!\!\!& \ni & \!\!\!\alpha_1+2\alpha_2+\ldots+2\alpha_{n-1}+\alpha_n, \\
S^p(G,X) \!\!\!& \not\ni & \!\!\!\alpha_1.
\end{eqnarray*}
\end{proposition}
Before proving the proposition, we remark that this kind of varieties appear also in rank $1$, namely case $9\mathsf C$, and in rank $2$, namely varieties $\mathbf{1}_{\mathit{rk}=2}$, $\mathbf{2}_{\mathit{rk}=2}$, $\mathbf{3}_{\mathit{rk}=2}$, $\mathbf{5}_{\mathit{rk}=2}$.
\begin{proof}
If $G\neq \Autz(X)$, then some of the border prime divisors of $(G,X)$ is not fixed. Hence at least one of the rank $1$ wonderful $G$-subvarieties must have a non fixed border prime divisor: call it $(G,X_{(1)})$. Since $(G,X)$ is indecomposable, definition \ref{def:prod} implies directly that $(G,X_{(1)})$ is contained in some rank $2$ wonderful $G$-subvariety having an indecomposable rank $2$ factor.

Call this factor $(G_1,Y)$: it has again a non-fixed border prime divisor, so it must appear in proposition \ref{prop:rank2}. If $(G_1,Y)$ is non-cuspidal, then $(G,X)$ falls into the family $\mathbf{5}_{\mathit{rk}=2}$, therefore our conditions are satisfied.

Otherwise, we have to look at the classification of rank $2$ wonderful varieties in \cite{Wa96}. We find out that varieties $\mathbf{2}_{\mathit{rk}=2}$ and $\mathbf{4}_{\mathit{rk}=2}$ cannot be subvarieties (nor factors of subvarieties) of any indecomposable bigger variety, so $(G_1,Y)$ is either $\mathbf{1}_{\mathit{rk}=2}$ or $\mathbf{3}_{\mathit{rk}=2}$: again, our conditions for $(G,X)$ are satisfied.

On the other hand, the classification in rank $2$ shows that if a variety satisfy our conditions, then the divisor associated to the spherical root $\alpha_1+2\alpha_2+\ldots+2\alpha_{n-1}+\alpha_n$ cannot be fixed, and this finishes the proof.
\end{proof}

We can describe a bit more precisely these exceptions in rank $> 2$: the classification in rank $2$ can be used easily to understand what are all wonderful varieties satisfying the conditions in proposition \ref{prop:highrank}. They might be regarded as composed by three families:
\begin{enumerate}
\item[$\mathbf{1}_{\mathit{rk}>2}$] any variety of rank $>2$ satisfying the conditions of $\mathbf{5}_{\mathit{rk}=2}$;
\item[$\mathbf{2}_{\mathit{rk}>2}$] any variety $(G,X)$ such that: $G$ has a simple factor $\mathsf{PSp}_{2n}$ with associated simple roots $\alpha_1,\ldots,\alpha_n$, and the invariants of $X$ satisfy:
\begin{eqnarray*}
\Sigma(G,X) \!\!\!& \supseteq &\!\!\! \left\{
\begin{array}{l}
\gamma_1=\alpha_1+2\alpha_2+\ldots+2\alpha_{n-1}+\alpha_n,\\
\gamma_2=\alpha_1,\\
\gamma_3=\alpha
\end{array}
\right\} \\
\A(G,X) \!\!\!& \ni & \!\!\!D \textrm{, such that $D$ is moved by both $\alpha_1$ and $\alpha$ }, 
\end{eqnarray*}
for $\alpha$ some simple root of $G$ not among $\alpha_1,\ldots,\alpha_n$.
\item[$\mathbf{3}_{\mathit{rk}>2}$] the rank $3$ variety $(G,X)$ where $G=\mathsf{PSp}_{2n}\times \mathsf{PSp}_{2m}$ and $X$ is uniquely determined by:
\[
\Sigma(G,X)= \left\{     
\begin{array}{l}
\gamma_1=\alpha_1+2\alpha_2+\ldots+2\alpha_{n_1-1}+\alpha_{n}, \\
\gamma_2=\alpha'_1+2\alpha'_2+\ldots+2\alpha'_{n_2-1}+\alpha'_{m}, \\
\gamma_3=\alpha_1+ \alpha'_1
\end{array} \right\}. 
\]
In this case, the variety $X$ has a decription similar to $\mathbf{1}_{\mathit{rk}=2}$, cf. subsection \ref{ssect:rank2}:
\[
X = \left\{ (E_1,E_2,M)  
\left|
\begin{array}{ll}
E_1 \in \Gr(2, \C^{2n}), \\
E_2 \in \Gr(2, \C^{2m}), \\ 
M \in \p(\Hom(E_1,E_2))
\end{array} \right.
\right\}
\]
\end{enumerate}

To conclude this section, we describe the varieties $(\Autz(X),X)$ for all these higher rank exceptions. Let $(G,X)$ be one of them: the group $G$ has one or more factors isomorphic to $\mathsf{PSp}$,
\[
G=G_1\times \mathsf{PSp}_{2n_1} \times \mathsf{PSp}_{2n_2}\times \cdots\times \mathsf{PSp}_{2n_k},
\]
where each $\mathsf{PSp}_{2n_i}$ (with simple roots $\alpha_1^i,\ldots,\alpha_{n_i}^i$) gives the prescribed spherical root $\alpha_1^i+2\alpha_2^i+\ldots+2\alpha_{n_i-1}^i+\alpha_{n_i}^i$ . Like in the previous subsection, it is straightforward to show that $\Autz(X)$ will differ from $G$ only in the factors $\mathsf{PSp}_{2n_i}$ which turn into $\mathsf{PSL}_{2n_i}$, and that the spherical systems of $X$ with respect to $G$ and to $\Autz(X)$ will differ only accordingly to this change.

Precisely: let us call $\beta_1^i,\ldots,\beta_{2n_i-1}^i$ the simple roots of $\mathsf{PSL}_{2n_i}$. Then, the spherical root $\alpha_1^i+2\alpha_2^i+\ldots+2\alpha_{n_i-1}^i+\alpha_{n_i}$ desappears in $\Sigma(\Autz(X),X)$; any other spherical root involving $\alpha_1^i$ in the support (such as for example $\alpha_1^i$ itself) remains the same, with $\beta_1^i$ taking the place of $\alpha_1^i$. All other spherical roots of $(G,X)$ appear unchanged in $(\Autz(X),X)$, and the sets $\A(G,X)$, $\A(\Autz(X),X)$ are exactly the same.

The set $S^p(\Autz(X),X)$ concides with $S^p(G,X)$ in what concerns the simple roots of $G_1$. Finally, $S^p(\Autz(X),X)$ ``behaves'' like $S^p(G,X)$ on the $\mathsf{PSp}$ factors of the group, in the sense that $S^p(G,X)$ contains all simple roots of $\mathsf{PSp}_{2n_i}$ except $\alpha_1^i$ and $\alpha_2^i$, and correspondingly $S^p(\Autz(X),X)$ contains all roots of $\mathsf{PSL}_{2n_i}$ except $\beta_1^i$ and $\beta_2^i$.

\section{Morphisms} \label{sect:morphisms}
\subsection{The set of all colors}
It possible to recover all colors and the values of the associated functionals, starting only from the spherical system of a wonderful variety $(G,X)$. Following \cite{Lu01} and its notations, we can identify each color in $\Delta(G,X)\setminus \A(G,X)$ with the simple roots it is moved by. This gives a disjoint union:
\[
\Delta(G,X) = \A(G,X) \cup \Delta^{a'}(G,X) \cup \Delta^b(G,X)
\]
The set $\Delta^{a'}(G,X)$ is in bijection with the set of simple roots $\alpha$ such that $2\alpha$ is a spherical root. For such a color $D$, we have $\rho_{G,X}(D)=\frac12 \alpha^\vee|_{\Xi(G,X)}$. The set $\Delta^b(G,X)$ is in bijection with the following set:
\[
\left.\left(S \setminus \left(\Sigma(G,X)\cup\frac12\Sigma(G,X)\cup S^p(G,X) \right) \right)\right/\sim
\]
where $\alpha\sim\beta$ if $\alpha=\beta$, or if $\alpha\perp\beta$ and $\alpha+\beta\in\Sigma(G,X)$ (or $\frac12(\alpha+\beta)\in\Sigma(G,X)$). For such a color $D$, the associated functional $\rho_{G,X}(D)$ is equal to $\alpha^\vee$, for $\alpha$ any representative of the $\sim$-equivalence class associated to $D$.

\subsection{A special class of smooth $G$-equivariant morphisms}
The theory of spherical varieties is used in \cite{Lu01} to study surjective $G$-equivariant morphisms with connected fibers between wonderful $G$-varieties.

Here it is enough to recall a particular case. Let $\Delta'$ be a subset of $\Delta(G,X)$, such that:
\[
\forall D\in\Delta',\forall\gamma\in\Sigma(G,X)\colon\;\; \langle D,\gamma \rangle \geq 0.
\]
Then there exist a unique (up to $G$-isomorphism) wonderful $G$-variety $X_{\Delta'}$, and a unique surjective $G$-equivariant map $f_{\Delta'}\colon X\to X_{\Delta'}$ with connected fibers, such that:
\[
\Delta' = \left\{ D\in\Delta(G,X) \;|\; f_{\Delta'}(D)=X_{\Delta'} \right\}.
\]
The morphism $f_{\Delta'}$ is smooth. The invariants of $X_{\Delta'}$ are the same of $X$, except for the following changes:
\begin{enumerate}
\item all spherical roots, where some color of $\Delta'$ was positive, disappear;
\item all colors of $\Delta'$ disappear; this may cause obvious changes in the sets $\A$ and $S^p$.
\end{enumerate}

A particular case occurs when $\Delta'$ contains only one element $\delta$: let us call {\em positive} such a color.

In \cite{Lu01} the induced maps $f_{\{\delta\}}$ were studied in details in the case where $\delta\in\A(G,X)$ and where $G$ has only components of type $\mathsf A$: it was called a {\em projective fibration}. Without these restrictions on $\delta$ and $G$ the analysis of {\em loc.cit.} cannot be always applied.

If we restrict ourselves to indecomposable varieties of rank at least $2$, we can reformulate our main theorem \ref{thm:main} in the following way: 
\begin{theorem}\label{thm:main2}
Let $(G,X)$ be an indecomposable wonderful variety of rank at least $2$. Then $\Autz(X)\neq G$ if and only if there exist a positive color $\delta\in\Delta(G,X)$ such that the variety $(G,X_{\{\delta\}})$ has a rank $1$ factor which is homogeneous under its full automorphism group.
\end{theorem}
\begin{proof}
The ``if'' part is an easy consequence of theorem \ref{thm:brion}, so it is interesting to prove it independently from theorem \ref{thm:main}, while the ``only if'' part follows from a case-by-case check of the proof of theorem \ref{thm:main}. Let us begin with the former.

Consider the variety $X_{\{\delta\}}$. The rank $1$ factor in the hypothesis implies that one the border divisors of $X_{\{\delta\}}$ is not fixed, in other words all colors of $X_{\{\delta\}}$ take non-negative values on the associated spherical root.

This border divisor corresponds to a border divisor of $X$, say $X^{(1)}$. But we have $\Delta(G,X)=\Delta(G,X_{\{\delta\}})\cup\{\delta\}$ and $\delta$ is never negative on spherical roots of $X$, so $X^{(1)}$ is not fixed, and $\Autz(X)\neq G$.

Now we proceed to the ``only if'' part. Here we must use the proof of theorem \ref{thm:main} and find the required $\delta$ in each of the exceptions listed in subsections \ref{ssect:rank2} and \ref{ssect:highrank}:
\begin{itemize}
\item[$\mathbf{1}_{\mathit{rk}=2}$] $\delta\in\Delta^b(G,X)$, moved by $\alpha_1$ and $\alpha_1'$;
\item[$\mathbf{2}_{\mathit{rk}=2}$] $\delta\in\Delta^{a'}(G,X)$, moved by $\alpha_1$;
\item[$\mathbf{3}_{\mathit{rk}=2}$] $\delta=D_1^+\in\A(G,X)$;
\item[$\mathbf{4}_{\mathit{rk}=2}$] $\delta\in\Delta^b(G,X)$, moved by $\alpha_4$;
\item[$\mathbf{5}_{\mathit{rk}=2}$] as for $\mathbf{1}_{\mathit{rk}=2}$;
\item[$\mathbf{1}_{\mathit{rk}>2}$] as for $\mathbf{1}_{\mathit{rk}=2}$;
\item[$\mathbf{2}_{\mathit{rk}>2}$] $\delta=D\in\A(G,X)$;
\item[$\mathbf{3}_{\mathit{rk}>2}$] $\delta\in\Delta^b(G,X)$, moved by $\alpha_1$ and $\alpha_1'$.
\end{itemize}
Once $\delta$ is found, a straightforward combinatorial verification finishes the proof.
\end{proof}

In the hypotheses of the above theorem, let us consider the variety $(\Autz(X),X)$ and its color $\delta$. It is easy to check case-by-case that $\delta$ is again positive under the action of $\Autz(X)$, thus it defines a smooth $\Autz(X)$-equivariant map:
\[
\tilde f_{\{\delta\}}\colon X\to \tilde X_{\{\delta\}}
\]
with the same properties as $f_{\{\delta\}}$. Again a simple case-by-case proof shows that $\tilde X_{\{\delta\}}=X_{\{\delta\}}$, and that $\Autz(X_{\{\delta\}})=\Autz(X)$. From the uniqueness property of $f_{\{\delta\}}$ and the fact that $\Autz(X)\supset G$, it follows also that $\tilde f_{\{\delta\}}=f_{\{\delta\}}$.

We may rephrase these conclusions in the following way: the group $\Autz(X_{\{\delta\}})$ is bigger than $G$, and all the elements of $\Autz(X_{\{\delta\}})\setminus G$ lift to $X$ via $f_{\{\delta\}}$.

It would be very interesting to have a direct proof of this lifting property of $f_{\{\delta\}}$. However this property fails if we try to replace the hypothesis of theorem \ref{thm:main2} with the weaker condition ``there exist $\delta$ such that $X_{\{\delta\}}$ has an automorphism group bigger than $G$''.

A counterexample is given by the rank $1$ cuspidal case 15 of \cite{Wa96}, where $G=\mathsf G_2$ and:
\begin{eqnarray*}
S^p(G,X) \!\!\!&=& \!\!\! \emptyset \\
\Sigma(G,X) \!\!\!&=& \!\!\!\{ \gamma_1=\alpha_1+\alpha_2 \} \\
\A(G,X) \!\!\! &=& \!\!\! \emptyset
\end{eqnarray*}
Here there are two colors $D_1,D_2\in\Delta^b(G,X)$, moved resp. by $\alpha_1,\alpha_2$. We can choose $\delta=D_2$, so that $X_{\{\delta\}}=\mathbf{3}_{\mathit{rk}=0}$. Thus $\Autz(X_{\{\delta\}}) \neq G$, but proposition \ref{prop:rank1} says that $\Autz(X)=G$.

\end{document}